\documentclass{amsart}
\usepackage{amssymb,amsmath}

\newtheorem{theorem}{Theorem}[section]
\newtheorem{lemma}[theorem]{Lemma}

\theoremstyle{definition}

\theoremstyle{remark}

\numberwithin{equation}{section}

\numberwithin{equation}{section}

\begin{document}

\title{On the Betti numbers of the tangent cones for Gorenstein Monomial Curves}

\author{P{\i}nar Mete}
\address{Department of Mathematics, 
Bal{\i}kesir University, Bal{\i}kesir, 10145 Turkey}
\email{pinarm@balikesir.edu.tr}

\subjclass{primary 13H10; secondary 13P10, 14H20}
\keywords{Gorenstein monomial curves, tangent cones, Betti numbers}
\date{\today}

\begin{abstract}
The aim of the article is to study the Betti numbers of the tangent cone of Gorenstein
monomial curves in affine 4-space. If $C_S$ is a non-complete intersection Gorenstein 
monomial curve whose tangent cone is Cohen-Macaulay, we  show that the possible 
Betti sequences are  (1,5,5,1), (1,5,6,2) and (1,6,8,3). 
\end{abstract}

\maketitle

\section{Introduction}
Let $S$ denote the numerical semigroup generated by the positive integers $n_1 <n_2 <\ldots < n_d$ with $gcd(n_1,\ldots,n_d)=1$.  
Consider the polynomial rings $R=k[x_1,\ldots,x_d]$ and $k[t]$ over the field $k$.  The semigroup ring $k[S]=k[t^{n_1},\ldots,t^{n_d}]$ 
is the $k-$subalgebra of $k[t]$.  $\varphi: R \to k[S] \subset k[t]$ with $\varphi (x_i)=t^{n_i}$ is the $k-$algebra homomorphism for 
$i=1,\ldots,d$ and its kernel $I_S$ is called the toric ideal of $S$.

Let $m=(t^{n_1},\ldots,t^{n_d})$ be the maximal ideal of the one-dimensional local ring $k[[t^{n_1},\ldots,t^{n_d}]]$. When $k$ is 
algebraically closed field, the semigroup ring $k[S]=k[t^{n_1},\ldots,t^{n_d}]$ is isomorphic to the coordinate ring $R/I_S$ of $C_S$
and the coordinate ring $gr_m(k[[t^{n_1},\ldots,t^{n_d}]])$ of  the tangent cone of $C_S$ at the origin  is isomorphic to the ring 
$R/I_{S_{*}}$. Here,  $I_{S_{*}}$ is generated by the polynomials $f_{*}$ which is the homogeneous summand of $f \in I_S$ and 
is called the defining ideal of the tangent cone of $C_S$. A monomial curve $C_S$ is Gorenstein if the associated local ring 
$k[[t^{n_1},\ldots,t^{n_d}]]$ is Gorenstein. $k[[t^{n_1},\ldots,t^{n_d}]]$ is Gorenstein if and only if the semigroup $S$ is 
symmetric \cite{kunz}. We recall that the numerical semigroup $S$ is symmetric if and only if for all $x \in \mathbb{Z}$ either 
$x \in S$ or $F(S)-x \in S$ , where $F(S)$ denote the Frobenius number of $S$.

Finding an explicit minimal free resolution of a standard $k-$algebra is one of the core areas in commutative algebra. Since it is 
very difficult to obtain a description of the differential in the resolution, we can get some information about the numerical invariants 
of the resolution such as Betti numbers. The i-th Betti number of an $R$-module $M$, $\beta_{i}(M)$, is the rank of the free modules 
appearing in the minimal free resolution of $M$ where $$0 \to R^{\beta_{k-1}} \to \ldots \to R^{\beta_{1}} \to R^{\beta_{0}} $$ and 
the Betti sequence of $M$, $\beta(M)$,  is $(\beta_{0}(M),\beta_{1}(M),\ldots,\beta_{k-1}(M))$.  Stamate \cite{stamate} gave a broad 
survey on the Betti numbers of the numerical semigroup rings and stated the problem of describing the Betti numbers and the 
minimal free resolution for the tangent cone when $S$ is 4-generated semigroup which is symmetric, or equivalently, $C_S$ is a 
Gorenstein monomial curve in affine 4-space [see  \cite{stamate}, Problem 9.9.]. In this paper, we solve this problem for 
Cohen-Macaulay tangent cone of a monomial curve in $\mathbb{A}^{4}(k)$ corresponding to a symmetric numerical semigroup. 
All computations have been done using {\footnotesize SINGULAR} \cite{greuel-pfister},\cite{singular}.

\section{The non-complete intersection Gorenstein Monomial Curves}
For the rest of the paper, we assume that $C_S$ is a Gorenstein non-complete intersection monomial curve in $\mathbb{A}^{4}$. 
Now, we recall Bresinsky's theorem, which gives the explicit description of the defining ideal of $C_S$.
\begin{theorem}\label{thm2.1}\cite{bresinsky} 
Let $C_S$ be a monomial curve having the parametrization
\[
x_1=t^{n_1}, \; x_2=t^{n_2}, \; x_3=t^{n_3}, \; x_4=t^{n_4}
\]
where $S=<n_{1},n_2,n_3,n_{4}>$ is a numerical semigroup minimally generated by $n_1,n_2,n_3,n_4$. The semigroup 
$<n_{1},n_{2},n_{3},n_{4}>$ is symmetric and $C_S$ is a non-complete intersection curve if and only if
$I_S$ is generated by the set
\[
\begin{split}
\{ &
f_{1}=x_{1}^{\alpha_{1}}-x_{3}^{\alpha_{13}}x_{4}^{\alpha_{14}},
f_{2}=x_{2}^{\alpha_{2}}-x_{1}^{\alpha_{21}}x_{4}^{\alpha_{24}},
f_{3}=x_{3}^{\alpha_{3}}-x_{1}^{\alpha_{31}}x_{2}^{\alpha_{32}}, \\
& f_{4}=x_{4}^{\alpha_{4}}-x_{2}^{\alpha_{42}}x_{3}^{\alpha_{43}},
f_{5}=x_{3}^{\alpha_{43}}x_{1}^{\alpha_{21}}-x_{2}^{\alpha_{32}}x_{4}^{\alpha_{14}}\}
\end{split}
\]
where the polynomials $f_{i}$'s are unique up to
isomorphism with $0 < \alpha_{ij} < \alpha_{j}$ with $\alpha_{i}n_i \in <n_1,\ldots,\hat{n}_i,\ldots,n_4>$ such that
$\alpha_{i}$'s are minimal for $1\leq i \leq 4$, where $\hat{n}_i$ denotes that
$n_i \notin <n_1,\ldots,\hat{n}_i,\ldots,n_4>$.
\end{theorem}

Theorem \ref{thm2.1} implies that for any non-complete intersection Gorenstein monomial curve $C_S$, the variables can be renamed to 
obtain generators exactly of the given form, and this means that there are six isomorphic possible
permutations which can be considered within three cases:

\begin{enumerate}
  \item $f_1=(1,(3,4))$
  \begin{enumerate}
     \item $f_2=(2,(1,4)), f_3=(3,(1,2)), f_4=(4,(2,3)), f_5=((1,3),(2,4))$
     \item $f_2=(2,(1,3)), f_3=(3,(2,4)), f_4=(4,(1,2)), f_5=((1,4),(2,3))$
   \end{enumerate}
   \item $f_1=(1,(2,3))$
   \begin{enumerate}
     \item $f_2=(2,(3,4)), f_3=(3,(1,4)), f_4=(4,(1,2)), f_5=((2,4),(1,3))$
     \item $f_2=(2,(1,4)), f_3=(3,(2,4)), f_4=(4,(1,3)), f_5=((1,3),(4,2))$
   \end{enumerate}
   \item $f_1=(1,(2,4))$
   \begin{enumerate}
     \item $f_2=(2,(1,3)), f_3=(3,(1,4)), f_4=(4,(2,3)), f_5=((1,2),(3,4))$
     \item $f_2=(2,(3,4)), f_3=(3,(1,2)), f_4=(4,(1,3)), f_5=((2,3),(1,4))$
    \end{enumerate}
\end{enumerate}

\vskip2mm\noindent Here, the notations $f_i\!\!=\!\!(i,(j,k))$ and $f_5\!\!=\!\!((i,j),(k,l))$ denote the generators 
$f_i\!\!=\!\!x_i^{\alpha_i}-x_j^{\alpha_{ij}}x_k^{\alpha_{ik}}$ and
$f_5\!\!=\!\!x_i^{\alpha_{ki}} x_j^{\alpha_{lj}}-x_k^{\alpha_{jk}}x_l^{\alpha_{il}}$.
Thus, given a Gorenstein monomial curve $C_S$, if we have the extra condition $n_1 < n_2 < n_3 < n_4$, then 
the generator set of its defining ideal $I_S$ is exactly given by one of these six permutations.

The study of the Cohen-Macaulayness of tangent cones of monomial curves constitutes an important 
problem \cite{arslan},\cite{arslan-katsabekis-nalbandiyan}. In \cite{arslan-mete}, Arslan and Mete determined the 
common arithmetic conditions satisfied by the generators of the defining ideals of  $C_S$ and under these 
conditions they found the generators of the tangent cone of $C_S$. In \cite{arslan-katsabekis-nalbandiyan}, they  
provided necessary and sufficient conditions for the Cohen-Macaulayness of the tangent cone of  $C_S$ in all six 
permutations and gave the following theorem:

\begin{theorem}\label{thm2.2} \cite{arslan-katsabekis-nalbandiyan} (1)
Suppose that $I_S$ is given as in the Case 1(a). Then $R/I_{S_{*}}$ is Cohen-Macaulay  if and only if 
$\alpha_2\leq \alpha_{21}+\alpha_{24}$.\\[2mm] (2) Suppose that $I_S$ is given as in Case 1(b). 
(i) Assume that $\alpha_{32} < \alpha_{42}$ and $\alpha_{14} \leq \alpha_{34}$. Then $R/I_{S_{*}}$ 
is Cohen-Macaulay  if and only if \\
1. $\alpha_2 \leq \alpha_{21}+\alpha_{23}$\\
2. $\alpha_{42} + \alpha_{13} \leq \alpha_{21}+\alpha_{34}$ and\\
3. $\alpha_{3} + \alpha_{13} \leq \alpha_{1}+\alpha_{32}+\alpha_{34}-\alpha_{14}$.\\[1mm]
(ii) Assume that $\alpha_{42} \leq \alpha_{32}$. Then $R/I_{S_{*}}$ is Cohen-Macaulay  if and only if \\
1. $\alpha_2 \leq \alpha_{21}+\alpha_{23}$\\
2. $\alpha_{42} + \alpha_{13} \leq \alpha_{21}+\alpha_{34}$ and\\
3. either $\alpha_{34} < \alpha_{14}$ and $\alpha_{3} + \alpha_{13} \leq \alpha_{21}+\alpha_{32}-\alpha_{42}+2\alpha_{34}$ \\ 
\indent\hspace{-2mm} or $\alpha_{14} \leq \alpha_{34}$ and  $\alpha_{3} + \alpha_{13} \leq \alpha_{1}+\alpha_{32}+\alpha_{34}-\alpha_{14}$.\\[2mm]
(3) Suppose that $I_S$ is given as in Case 2(a). (i) Assume that $\alpha_{24} < \alpha_{34}$ and $\alpha_{13} \leq \alpha_{23}$. Then $R/I_{S_{*}}$ 
is Cohen-Macaulay  if and only if \\
1. $\alpha_3 \leq \alpha_{31}+\alpha_{34}$\\
2. $\alpha_{12} + \alpha_{34} \leq \alpha_{41}+\alpha_{23}$ and\\
3. $\alpha_{2} + \alpha_{12} \leq \alpha_{1}+\alpha_{23}-\alpha_{13}+\alpha_{24}$.\\[1mm]
(ii) Assume that $\alpha_{34} \leq \alpha_{24}$. Then $R/I_{S_{*}}$ is Cohen-Macaulay  if and only if \\
1. $\alpha_3 \leq \alpha_{31}+\alpha_{34}$\\
2. $\alpha_{12} + \alpha_{34} \leq \alpha_{41}+\alpha_{23}$ and\\
3. either $\alpha_{23} < \alpha_{13}$ and $\alpha_{2} + \alpha_{12} \leq \alpha_{41}+2\alpha_{23}+\alpha_{24}-\alpha_{34}$ \\ 
\indent\hspace{-2mm} or $\alpha_{13} \leq \alpha_{23}$ and  $\alpha_{2} + \alpha_{12} \leq \alpha_{1}+\alpha_{23}-\alpha_{13}+\alpha_{24}$.\\[2mm]
(4) Suppose that $I_S$ is given as in Case 2(b). (i) Assume that $\alpha_{34} < \alpha_{24}$ and $\alpha_{12} \leq \alpha_{32}$. Then $R/I_{S_{*}}$ 
is Cohen-Macaulay  if and only if \\
1. $\alpha_2 \leq \alpha_{21}+\alpha_{24}$ and \\ 
2. $\alpha_{3} + \alpha_{13} \leq \alpha_{1}+\alpha_{32}-\alpha_{12}+\alpha_{34}$.\\[1mm]
(ii) Assume that $\alpha_{24} \leq \alpha_{34}$. Then $R/I_{S_{*}}$ is Cohen-Macaulay if and only if \\
1. $\alpha_2 \leq \alpha_{21}+\alpha_{24}$ and\\
2. either $\alpha_{32} < \alpha_{12}$ and $\alpha_{3} + \alpha_{13} \leq \alpha_{41}+2\alpha_{32}+\alpha_{34}-\alpha_{24}$ \\ 
\indent\hspace{-2mm} or $\alpha_{12} \leq \alpha_{32}$ and  $\alpha_{3} + \alpha_{13} \leq \alpha_{1}+\alpha_{32}-\alpha_{12}+\alpha_{34}$.\\[2mm]
(5) Suppose that $I_S$ is given as in the Case 3(a). Then $R/I_{S_{*}}$ is Cohen-Macaulay tangent cone  if and only if 
$\alpha_2\leq \alpha_{21}+\alpha_{23}$ and $\alpha_3\leq \alpha_{31}+\alpha_{34}$.\\[2mm] 
(6) Suppose that $I_S$ is given as in Case 3(b). (i) Assume that $\alpha_{23} < \alpha_{43}$ and 
$\alpha_{14} \leq \alpha_{24}$. Then $R/I_{S_{*}}$ is Cohen-Macaulay  if and only if \\
1. $\alpha_{12} + \alpha_{43} \leq \alpha_{31}+\alpha_{24}$ and\\
2. $\alpha_{2} + \alpha_{12} \leq \alpha_{1}+\alpha_{23}+\alpha_{24}-\alpha_{14}$.\\[1mm]
(ii) Assume that $\alpha_{43} \leq \alpha_{23}$. Then $R/I_{S_{*}}$ is Cohen-Macaulay  if and only if \\
1. $\alpha_{12} + \alpha_{43} \leq \alpha_{31}+\alpha_{24}$ and\\
2. either $\alpha_{24} < \alpha_{14}$ and $\alpha_{2} + \alpha_{12} \leq \alpha_{31}+2\alpha_{24}+\alpha_{23}-\alpha_{43}$ \\
\indent\hspace{-2mm} or 
$\alpha_{14} \leq \alpha_{24}$ and  $\alpha_{2} + \alpha_{12} \leq \alpha_{1}+\alpha_{23}+\alpha_{24}-\alpha_{14}$.\\
 \end{theorem}
\vspace{-3mm}Recently, Katsabekis gave the remaining standard bases for $I_S$ in \cite{katsabekis} using above conditions for the 
Cohen-Macaulayness of the tangent cone of $C_S$.  

\section{Betti Sequences of Cohen-Macaulay Tangent Cones }

In this section, we determine the Betti sequences of the tangent cone of $C_S$ whose tangent cone is Cohen-Macaulay. 
Buchsbaum-Eisenbud criterion \cite{buchsbaum-eisenbud} and the following Lemma (see also \cite{herzog-stamate})
will be used in the all proofs.  

\begin{lemma} \label{lemma3.1}\cite{sahin-sahin}
Assume that the multiplicity of the monomial curve is $n_i$. Suppose that the k-algebra homomorphism 
$\pi_{i} : R \to \overline{R}=k[x_1,\ldots,\overline{x_i},\ldots,x_d]$ is defined by $\pi_{i} (x_i)=\overline{x_i}=0$ and 
$\pi_{j}(x_j)=x_j$ for $i \neq j$, and set $\overline{I}=\pi_{i}(I_{S_{*}})$. If the tangent cone 
$gr_m(k[[t^{n_1},\ldots,t^{n_d}]])$ is Cohen-Macaulay, then the Betti sequences of $gr_m(k[[t^{n_1},\ldots,t^{n_d}]])$ 
and of $\overline{R}/\overline{I}$ are the same.
\end{lemma}

This is a very effective result to reduce the number of cases for finding the Betti numbers of the tangent 
cones.


\vspace{4mm}\noindent\textbf{Case 1(a) :}
Suppose that $I_S$ is minimally generated by the polynomials
$$f_{1}=x_{1}^{\alpha_{1}}-x_{3}^{\alpha_{13}}x_{4}^{\alpha_{14}},
\hspace{7mm}f_{2}=x_{2}^{\alpha_{2}}-x_{1}^{\alpha_{21}}x_{4}^{\alpha_{24}},
\hspace{7mm}f_{3}=x_{3}^{\alpha_{3}}-x_{1}^{\alpha_{31}}x_{2}^{\alpha_{32}},$$ \\[-9mm]
$$f_{4}=x_{4}^{\alpha_{4}}-x_{2}^{\alpha_{42}}x_{3}^{\alpha_{43}}, 
\hspace{7mm}f_{5}=x_{1}^{\alpha_{21}}x_{3}^{\alpha_{43}}-x_{2}^{\alpha_{32}}x_{4}^{\alpha_{14}}$$

\noindent where $\alpha_{1}=\alpha_{21}+\alpha_{31}$,  $\alpha_{2}=\alpha_{32}+\alpha_{42}$,
 $\alpha_{3}=\alpha_{13}+\alpha_{43}$,  $\alpha_{4}=\alpha_{14}+\alpha_{24}$. 
 The condition $n_1<n_2<n_3<n_4$ implies
$\alpha_1>\alpha_{13}+\alpha_{14}, \; \alpha_4<\alpha_{42}+\alpha_{43}$ and $\alpha_3<\alpha_{31}+\alpha_{32}.$\\

$C_S$ has  Cohen-Macaulay tangent cone only under the  condition 
$\alpha_2 \leq \alpha_{21}+\alpha_{24}$ by \cite{arslan-katsabekis-nalbandiyan}.
Mete and Zengin   \cite{mete-zengin} gave the explicit minimal free resolution for the tangent cone of 
$C_S$ and showed that the Betti sequence of its tangent cone is $(1,5,6,2)$.

 
\vspace{4mm}\noindent\textbf{Case 1(b) :} Suppose that $I_S$ is minimally generated by the polynomials
$$f_{1}=x_{1}^{\alpha_{1}}-x_{3}^{\alpha_{13}}x_{4}^{\alpha_{14}},
\hspace{7mm}f_{2}=x_{2}^{\alpha_{2}}-x_{1}^{\alpha_{21}}x_{3}^{\alpha_{23}},
\hspace{7mm}f_{3}=x_{3}^{\alpha_{3}}-x_{2}^{\alpha_{32}}x_{4}^{\alpha_{34}}, $$ \\[-9mm]
$$f_{4}=x_{4}^{\alpha_{4}}-x_{1}^{\alpha_{41}}x_{2}^{\alpha_{42}},
\hspace{7mm}f_{5}=x_{1}^{\alpha_{21}}x_{4}^{\alpha_{34}}-x_{2}^{\alpha_{42}}x_{3}^{\alpha_{13}}$$

\noindent where $\alpha_{1}=\alpha_{21}+\alpha_{41}$,  $\alpha_{2}=\alpha_{32}+\alpha_{42}$,
 $\alpha_{3}=\alpha_{13}+\alpha_{23}$,  $\alpha_{4}=\alpha_{14}+\alpha_{34}$.
The condition $n_1<n_2<n_3<n_4$ implies
$\alpha_1>\alpha_{13}+\alpha_{14}$ and $ \alpha_4<\alpha_{41}+\alpha_{42}$. Under the restriction 
$\alpha_{2} \leq \alpha_{21}+\alpha_{23}$ by Theorem \ref{thm2.2} and one possible condition 
$\alpha_{3} \leq \alpha_{32}+\alpha_{34}$, the Betti sequences for 
the Cohen-Macaulay tangent cone of $C_S$ are $(1,5,5,1)$ and $(1,5,6,2)$, as given in \cite{mete-zengin}.
Here, the other possible condition $\alpha_{3} >  \alpha_{32}+\alpha_{34}$ will be considered. 

\begin{theorem}\label{thm3.2} The Betti sequence of the tangent cone $R/I_{S_*}$ of $C_S$ is $(1,6,8,3)$, 
if $I_S$ is given as in Case 1(b) when $\alpha_{3} >  \alpha_{32}+\alpha_{34}$.
\end{theorem}

\begin{proof} Suppose that $I_S$ is given as in the Case 1(b) when $\alpha_{3} >  \alpha_{32}+\alpha_{34}$. 

\noindent (i)  By Proposition 2.7 in \cite{katsabekis}, if $\alpha_{32} < \alpha_{42}$ and $\alpha_{14} \leq \alpha_{34}$, then,
$$G=\{f_{1},f_{2},f_{3}, f_{4},f_{5},f_{6}=x_{3}^{\alpha_{3}+\alpha_{13}}-x_1^{\alpha_{1}}x_{2}^{\alpha_{32}}x_{4}^{\alpha_{34}-\alpha_{14}}\},$$
(ii) By Proposition 2.9 in \cite{katsabekis},

\indent (1) if $\alpha_{42} \leq \alpha_{32}$ and $\alpha_{34} < \alpha_{14}$, then 
$$G=\{f_{1},f_{2},f_{3}f_{4},f_{5},f_{6}=x_{3}^{\alpha_{3}+\alpha_{13}}-x_1^{\alpha_{21}}x_{2}^{\alpha_{32}-\alpha_{42}}x_{4}^{2\alpha_{34}}\},$$

\indent (2) if $\alpha_{42} \leq \alpha_{32}$ and $\alpha_{14} \leq \alpha_{34}$, then 
$$G=\{f_{1},f_{2},f_{3},f_{4},f_{5},f_{6}=x_{3}^{\alpha_{3}+\alpha_{13}}-x_1^{\alpha_{1}}x_{2}^{\alpha_{32}}x_{4}^{\alpha_{34}-\alpha_{14}}\}$$
\noindent are standard bases for $I_S$.
 Since $\overline{I}=\pi_{i}(I_{S_{*}})$ which sends $x_1$ to 0, the generators of the defining ideal of $\overline{I}$ is generated by 
$$G_{*}=(x_3^{\alpha_{13}}x_4^{\alpha_{14}},x_2^{\alpha_2},x_2^{\alpha_{32}}x_4^{\alpha_{34}},x_4^{\alpha_{4}},x_2^{\alpha_{42}}x_3^{\alpha_{13}},x_3^{\alpha_{3}+\alpha_{13}}).$$

\noindent  Now, consider the case (i). Since the Betti sequences of   $R/I_{S_*}$ and $\overline{R}/\overline{I}$ are the same which 
follows from Lemma \ref{lemma3.1}, we will show that the sequence,
\begin{center}
$ 0 \rightarrow R^3 \xrightarrow{\phi_3} R^8 \xrightarrow{\phi_2} R^6\xrightarrow{\phi_1} R^1  \rightarrow 0 $
\end{center}
\noindent is a minimal free resolution of $\overline{R}/\overline{I}$, where

{\footnotesize
$$
\phi_1=
\begin{pmatrix}
x_3^{\alpha_{13}}x_4^{\alpha_{14}} &
x_{2}^{\alpha_{2}}&
x_{2}^{\alpha_{32}}x_{4}^{\alpha_{34}}&
x_{4}^{\alpha_{4}}&
x_{2}^{\alpha_{42}}x_{3}^{\alpha_{13}}&
x_{3}^{\alpha_{3}+\alpha_{13}} 
\end{pmatrix},
$$\\[-7mm]

$$
\phi_2=
\begin{pmatrix}
0  & -x_2^{\alpha_{32}}x_4^{\alpha_{34}-\alpha_{14}} & 0  &  -x_4^{\alpha_{34}} & x_2^{\alpha_{42}} & x_3^{\alpha_{3}}  & 0 &  0 \\[0.5mm]
0 & 0 &  -x_4^{\alpha_{34}} & 0 & 0 & 0 & x_3^{\alpha_{13}} & 0  \\[0.5mm]
x_4^{\alpha_{14}} &  x_3^{\alpha_{13}}  & x_2^{\alpha_{42}} & 0  & 0  & 0 & 0 & 0 \\[0.5mm]
-x_2^{\alpha_{32}} & 0 & 0 & x_3^{\alpha_{13}} & 0 & 0 & 0 & 0 \\[0.5mm]
0 & 0 & 0 & 0 & -x_4^{\alpha_{14}} & 0 &  -x_2^{\alpha_{32}} &  x_3^{\alpha_{3}} \\[0.5mm]
0 & 0 & 0 & 0 & 0 & -x_4^{\alpha_{14}}& 0 & -x_2^{\alpha_{42}} \\
\end{pmatrix},
$$\\[-7mm]

$$
\phi_3=
\begin{pmatrix}
x_{3}^{\alpha_{13}} & 0 & 0\\[0.5mm]
-x_{4}^{\alpha_{14}} &   x_{2}^{\alpha_{42}} & 0 \\[0.5mm]
0 & -x_{3}^{\alpha_{13}} & 0 \\[0.5mm]
x_{2}^{\alpha_{32}} & 0  & 0\\[0.5mm]
0 &  x_{2}^{\alpha_{32}}x_{4}^{\alpha_{34}-\alpha_{14}}& x_{3}^{\alpha_{3}}\\[0.5mm]
0 & 0  & -x_{2}^{\alpha_{42}} \\[0.5mm]
0 & x_{4}^{\alpha_{34}} & 0  \\[0.5mm]
0 & 0 & x_{4}^{\alpha_{14}}
\end{pmatrix}.
$$
}

\noindent Since $\phi_1\phi_2=\phi_2\phi_3=0$, the sequence above is a complex. 
 One can easily check that $rank(\phi_1)=1$,  $rank(\phi_2)=5$ and  $rank(\phi_3)=3$.
As the 5-minors of $\phi_2$, we get  
$x_3^{2\alpha_{3}+3\alpha_{13}}$ by deleting the row 6, and the columns 1, 3, 5,  and 
$-x_4^{2\alpha_{4}+\alpha_{14}}$ by deleting the row 4 and the columns 2, 7, 8. These two determinants are 
relatively prime, so $I(\phi_2)$ contains a regular sequence of length 2. As the 3-minors of $\phi_3$, we have 
$x_2^{\alpha_{2}+\alpha_{42}}$ by deleting the rows 1, 3, 5, 7, 8,  and  $-x_3^{\alpha_{3}+2\alpha_{13}}$ by 
deleting the rows 2, 4, 6, 7, 8, and finally, $-x_4^{\alpha_{4}+\alpha_{14}}$ by deleting the rows 1, 3, 4, 5, 6 
Since these are relatively prime, $I(\phi_3)$ contains a regular sequence of length 3.

The other cases (ii)-(1) and (2) can be done similarly and are, hence, omitted.
\end{proof}

 
\vspace{5mm}\noindent\textbf{Case 2(a) :} Suppose that $I_S$ is minimally generated by the polynomials
$$f_{1}=x_{1}^{\alpha_{1}}-x_{2}^{\alpha_{12}}x_{3}^{\alpha_{13}},
\hspace{7mm}f_{2}=x_{2}^{\alpha_{2}}-x_{3}^{\alpha_{23}}x_{4}^{\alpha_{24}},
\hspace{7mm}f_{3}=x_{3}^{\alpha_{3}}-x_{1}^{\alpha_{31}}x_{4}^{\alpha_{34}}, $$ \\[-9mm]
$$f_{4}=x_{4}^{\alpha_{4}}-x_{1}^{\alpha_{41}}x_{2}^{\alpha_{42}},  
\hspace{7mm}f_{5}=x_{1}^{\alpha_{41}}x_{3}^{\alpha_{23}}-x_{2}^{\alpha_{12}}x_{4}^{\alpha_{34}}$$

\noindent where $\alpha_{1}=\alpha_{31}+\alpha_{41}$,  $\alpha_{2}=\alpha_{12}+\alpha_{42}$,
 $\alpha_{3}=\alpha_{13}+\alpha_{23}$,  $\alpha_{4}=\alpha_{24}+\alpha_{34}.$
The condition $n_1<n_2<n_3<n_4$ implies
$\alpha_1>\alpha_{12}+\alpha_{13}$, \;$\alpha_2>\alpha_{23}+\alpha_{24}$ and 
$ \alpha_4<\alpha_{41}+\alpha_{42}$.  Note  that  the assumption $ \alpha_3 \leq \alpha_{31}+\alpha_{34}$ 
is one of the necessary and sufficient conditions for the Cohen-Macaulayness of the tangent cone 
by Theorem \ref{thm2.2}.

\begin{theorem}\label{thm3.3}  The Betti sequence of the tangent cone $R/I_{S_*}$ of $C_S$ is $(1,6,8,3)$, 
if $I_S$ is given as in Case 2(a).
\end{theorem}

\begin{proof} Suppose that $I_S$ is given as in the Case 2(a). 

\noindent (i) By Proposition 2.11 in \cite{katsabekis}, if $\alpha_{24} < \alpha_{34}$ and 
$\alpha_{13} \leq \alpha_{23}$, then
$$G=\{f_{1},f_{2},f_{3},f_{4},f_{5},f_{6}=x_{2}^{\alpha_{2}+\alpha_{12}}-x_1^{\alpha_{1}}x_{3}^{\alpha_{23}-\alpha_{13}}x_{4}^{\alpha_{24}}\}$$

\noindent (ii) By Proposition 2.13 in \cite{katsabekis},

\indent (1) if $\alpha_{34} \leq \alpha_{24}$ and $\alpha_{23} < \alpha_{13}$, then 
$$G=\{f_{1},f_{2},f_{3},f_{4},f_{5},f_{6}=x_{2}^{\alpha_{2}+\alpha_{12}}-x_1^{\alpha_{41}}x_{3}^{2\alpha_{23}}x_{4}^{\alpha_{24}-\alpha_{34}}\}$$

\indent (2) if $\alpha_{34} \leq \alpha_{24}$ and $\alpha_{13} \leq \alpha_{23}$, then 
$$G=\{f_{1},f_{2},f_{3},f_{4},f_{5},f_{6}=x_{2}^{\alpha_{2}+\alpha_{12}}-x_1^{\alpha_{1}}x_{3}^{\alpha_{23}-\alpha_{13}}x_{4}^{\alpha_{24}}\}$$

\noindent are standard bases for $I_S$.
Since $\overline{I}=\pi_{i}(I_{S_*})$ which sends $x_1$ to 0, the generators of the defining ideal of $\overline{I}$ is generated by 
$$G_{*}=(x_2^{\alpha_{12}}x_3^{\alpha_{13}},x_3^{\alpha_{23}}x_4^{\alpha_{24}},x_3^{\alpha_3},x_4^{\alpha_{4}},x_2^{\alpha_{12}}x_4^{\alpha_{34}},x_2^{\alpha_{2}+\alpha_{12}}).$$

\noindent  Now, consider the case (i). Since the Betti sequences of   $R/I_{S_*}$ and $\overline{R}/\overline{I}$ are the same which 
follows from Lemma \ref{lemma3.1}, we will show that the sequence,
\begin{center}
$ 0 \rightarrow R^3 \xrightarrow{\phi_3} R^8 \xrightarrow{\phi_2} R^6\xrightarrow{\phi_1} R^1  \rightarrow 0 $
\end{center}
\noindent is a minimal free resolution of $\overline{R}/\overline{I}$, where

{\footnotesize
$$
\phi_1=
\begin{pmatrix}
x_2^{\alpha_{12}}x_3^{\alpha_{13}} &
x_3^{\alpha_{23}}x_4^{\alpha_{24}}&
x_{3}^{\alpha_{3}}&
x_{4}^{\alpha_{4}}&
x_{2}^{\alpha_{12}}x_{4}^{\alpha_{34}}&
x_{2}^{\alpha_{2}+\alpha_{12}} 
\end{pmatrix},
$$\\[-7mm]

$$
\phi_2=
\begin{pmatrix}
x_3^{\alpha_{23}}  & x_2^{\alpha_{23}-\alpha_{13}}x_4^{\alpha_{24}} &   x_4^{\alpha_{34}} & x_2^{\alpha_{2}} &  0  & 0  & 0 &  0 \\[0.5mm]
0 & -x_2^{\alpha_{12}} & 0  & 0 &  -x_3^{\alpha_{13}}  &  x_4^{\alpha_{34}} & 0  & 0 \\[0.5mm]
-x_2^{\alpha_{12}} &  0  & 0 & 0  & x_4^{\alpha_{24}}  & 0 & 0 & 0 \\[0.5mm]
0 & 0 & 0 & 0 & 0 & -x_3^{\alpha_{23}} & -x_2^{\alpha_{12}} & 0  \\[0.5mm]
0 & 0 &  -x_3^{\alpha_{13}} & 0 & 0 & 0 &  x_4^{\alpha_{24}} &  x_2^{\alpha_{2}} \\[0.5mm]
0 & 0 & 0 & -x_3^{\alpha_{13}} & 0 & 0 & 0 &-x_4^{\alpha_{34}} \\
\end{pmatrix},
$$\\[-7mm]

$$
\phi_3=
\begin{pmatrix}
x_{4}^{\alpha_{24}} & 0 & 0\\[0.5mm]
-x_{3}^{\alpha_{13}} &   x_{4}^{\alpha_{34}} & 0 \\[0.5mm]
0 & -x_{3}^{\alpha_{23}-\alpha_{13}}x_4^{\alpha_{24}} & x_2^{\alpha_{2}} \\[0.5mm]
0 & 0  &  -x_{4}^{\alpha_{34}} \\[0.5mm]
x_{2}^{\alpha_{12}} &  0 & 0 \\[0.5mm]
0 & x_{2}^{\alpha_{12}} &  0 \\[0.5mm]
0 & -x_{3}^{\alpha_{23}} & 0  \\[0.5mm]
0 & 0 & x_{3}^{\alpha_{13}}
\end{pmatrix}.
$$
}

\noindent Since $\phi_1\phi_2=\phi_2\phi_3=0$, the  sequence above is a complex. 
One can easily check that $rank(\phi_1)=1$,  $rank(\phi_2)=1$  and $rank(\phi_3)=3$. 

As the 5-minors of $\phi_2$ , we have  $-x_3^{2\alpha_{3}+\alpha_{13}}$ by removing the columns  2, 7, 8 
and the row 3, and  $-x_4^{2{\alpha_{4}+\alpha_{34}}}$ by removing the columns 1, 2, 4 and the row 4. 
These two determinants are relatively prime, so $I(\phi_2)$ contains a regular sequence of length 2. 
As the 3-minors of $\phi_3$, we get  $x_2^{\alpha_{2}+2\alpha_{12}}$ by removing the rows 1, 2, 4, 7, 8, and  
$x_3^{\alpha_{3}+\alpha_{13}}$ by removing the rows 1, 3, 4, 5, 6, and finally $-x_4^{\alpha_{4}+\alpha_{34}}$
by removing the rows 3, 5, 6, 7, 8. These three determinants constitute a regular sequence. 
Thus, $I(\phi_3)$ contains a regular sequence of length 3.
 
The other cases (ii)-(1) and (2) can be done similarly and are, hence, omitted.
\end{proof}


\vspace{5mm}\noindent\textbf{Case 2(b) :} Suppose that $I_S$ is minimally generated by the polynomials
$$f_{1}=x_{1}^{\alpha_{1}}-x_{2}^{\alpha_{12}}x_{3}^{\alpha_{13}},
\hspace{7mm}f_{2}=x_{2}^{\alpha_{2}}-x_{1}^{\alpha_{21}}x_{4}^{\alpha_{24}},
\hspace{7mm}f_{3}=x_{3}^{\alpha_{3}}-x_{2}^{\alpha_{32}}x_{4}^{\alpha_{34}}, $$ \\[-9mm]
$$f_{4}=x_{4}^{\alpha_{4}}-x_{1}^{\alpha_{41}}x_{3}^{\alpha_{43}},
\hspace{7mm}f_{5}=x_{1}^{\alpha_{41}}x_{2}^{\alpha_{32}}-x_{3}^{\alpha_{13}}x_{4}^{\alpha_{24}}.$$

\noindent where $\alpha_{1}=\alpha_{21}+\alpha_{41}$,  $\alpha_{2}=\alpha_{12}+\alpha_{32}$,
 $\alpha_{3}=\alpha_{13}+\alpha_{43}$,  $\alpha_{4}=\alpha_{24}+\alpha_{34}.$
The condition $n_1<n_2<n_3<n_4$ implies
$\alpha_1>\alpha_{12}+\alpha_{13}$ and $ \alpha_4<\alpha_{41}+\alpha_{43}$.

Under the restriction $\alpha_{2} \leq \alpha_{21}+\alpha_{24}$ 
by Theorem \ref{thm2.2} and one possible condition $\alpha_{3} \leq \alpha_{32}+\alpha_{34}$, 
the Betti sequences for the Cohen-Macaulay tangent cone of $C_S$ are $(1,5,5,1)$ and $(1,5,6,2)$, 
as given in \cite{mete-zengin}. Here, the other possible condition $\alpha_{3} >  \alpha_{32}+\alpha_{34}$ 
will be considered. 

\begin{theorem}\label{thm3.4} The Betti sequence of the tangent cone $R/I_{S_*}$ of $C_S$ is $(1,6,8,3)$, 
if $I_S$ is given as in Case 2(b) when $\alpha_{3} >  \alpha_{32}+\alpha_{34}$.
\end{theorem}

\begin{proof} Suppose that $I_S$ is given as in the Case 2(b). 

\noindent (i) By Proposition 2.16 in \cite{katsabekis},
if $\alpha_{34} < \alpha_{24}$ and $\alpha_{12} \leq \alpha_{32}$, then
$$G=\{f_{1},f_{2},f_{3},f_{4},f_{5},f_{6}=x_{3}^{\alpha_{3}+\alpha_{13}}-x_1^{\alpha_{1}}x_{2}^{\alpha_{32}-\alpha_{12}}x_{4}^{\alpha_{34}}\}$$

\noindent (ii) By Proposition 2.18 in \cite{katsabekis}, 
  
\indent (1) if $\alpha_{24} \leq \alpha_{34}$ and $\alpha_{32} < \alpha_{12}$, then 
$$G=\{f_{1},f_{2},f_{3},f_{4},f_{5},f_{6}=x_{3}^{\alpha_{3}+\alpha_{13}}-x_1^{\alpha_{41}}x_{2}^{2\alpha_{32}}x_{4}^{\alpha_{34}-\alpha_{24}}\},$$

\indent (2) if $\alpha_{24} \leq \alpha_{34}$ and $\alpha_{12} \leq \alpha_{32}$, then 
$$G=\{f_{1},f_{2},f_{3},f_{4},f_{5},f_{6}=x_{3}^{\alpha_{3}+\alpha_{13}}-x_1^{\alpha_{1}}x_{2}^{\alpha_{32}-\alpha_{12}}x_{4}^{\alpha_{34}}\}$$
\noindent are standard bases for $I_S$.
Since $\overline{I}=\pi_{i}(I_{S_*})$ which sends $x_1$ to 0, then the generators of the defining ideal of $\overline{I}$ is generated by 

$$G_{*}=(x_2^{\alpha_{12}}x_3^{\alpha_{13}},  x_2^{\alpha_{2}}, x_2^{\alpha_{32}}x_4^{\alpha_{34}}, x_4^{\alpha_{4}}, x_3^{\alpha_{13}}x_4^{\alpha_{24}},x_3^{\alpha_{3}+\alpha_{13}}).$$

\noindent  Now, consider the case (i). Since the Betti sequences of   $R/I_{S_*}$ and $\overline{R}/\overline{I}$ are the same which 
follows from Lemma \ref{lemma3.1}, we will show that the sequence,
\begin{center}
$ 0 \rightarrow R^3 \xrightarrow{\phi_3} R^8 \xrightarrow{\phi_2} R^6\xrightarrow{\phi_1} R^1  \rightarrow 0 $
\end{center}
\noindent is a minimal free resolution of $\overline{R}/\overline{I}$, where

{\footnotesize
$$
\phi_1=
\begin{pmatrix}
x_2^{\alpha_{12}}x_3^{\alpha_{13}} &
x_2^{\alpha_{2}}&
x_2^{\alpha_{32}}x_4^{\alpha_{34}}&
x_{4}^{\alpha_{4}}&
x_{3}^{\alpha_{13}}x_{4}^{\alpha_{24}}&
x_{3}^{\alpha_{3}+\alpha_{13}} 
\end{pmatrix},
$$\\[-7mm]

$$
\phi_2=
\begin{pmatrix}
x_4^{\alpha_{24}}  & x_2^{\alpha_{32}} &   x_3^{\alpha_{3}} & 0  &  0  & 0  & 0 &  0 \\[0.5mm]
0 & -x_3^{\alpha_{13}} & 0  & 0 &  0  &  0  & x_4^{\alpha_{34}}  & 0 \\[0.5mm]
0  &  0  & 0 & 0  & -x_3^{\alpha_{13}} &  0 & -x_2^{\alpha_{12}}   & -x_4^{\alpha_{24}}  \\[0.5mm]
0 & 0 & 0 &-x_3^{\alpha_{13}} & 0  & 0 & 0 &  x_2^{\alpha_{32}}  \\[0.5mm]
-x_2^{\alpha_{12}} & 0 & 0  &  x_4^{\alpha_{34}}& x_2^{\alpha_{32}}x_4^{\alpha_{34}-\alpha_{24}}  &  x_3^{\alpha_{3}} & 0 & 0  \\[0.5mm]
0 & 0 &  -x_2^{\alpha_{12}} &0  & 0 & -x_4^{\alpha_{24}} & 0 & 0  \\
\end{pmatrix},
$$\\[-7mm]

$$
\phi_3=
\begin{pmatrix}
x_{3}^{\alpha_{3}} & -x_{2}^{\alpha_{32}}x_4^{\alpha_{34}-\alpha_{24}} & 0\\[0.5mm]
0  &   x_{4}^{\alpha_{34}} & 0 \\[0.5mm]
-x_{4}^{\alpha_{24}}&  0  & 0 \\[0.5mm]
0 & 0  &  x_{2}^{\alpha_{32}} \\[0.5mm]
0  &  -x_{2}^{\alpha_{12}} & -x_{4}^{\alpha_{24}} \\[0.5mm]
x_{2}^{\alpha_{12}} & 0 &  0 \\[0.5mm]
0 & x_{3}^{\alpha_{13}} & 0  \\[0.5mm]
0 & 0 & x_{3}^{\alpha_{13}}
\end{pmatrix}.
$$
}

\noindent Since $\phi_1\phi_2=\phi_2\phi_3=0$, the sequence above is a complex. 
One can easily check that $rank(\phi_1)=1$, $rank(\phi_2)=5$ and  $rank(\phi_3)=3$. 
As the 5-minors of $\phi_2$, we get $-x_3^{2\alpha_{3}+3\alpha_{13}}$ by removing
the row 6 and the columns 1, 7, 8, and  $x_4^{2{\alpha_{4}+\alpha_{24}}}$ by removing
the row 4 and  the columns 2, 3, 5. These two determinants are relatively prime, 
so $I(\phi_2)$ contains a regular sequence of length 2. As the 3-minors of $\phi_3$, 
we have $-x_2^{\alpha_{2}+\alpha_{12}}$ by removing the rows 1, 2, 3, 7, 8, and 
 $x_3^{\alpha_{3}+2\alpha_{13}}$ by removing the rows 2, 3, 4, 5, 6, and finally, 
 $-x_4^{\alpha_{4}+\alpha_{24}}$ by removing the rows 1, 4, 6, 7, 8.  These three 
 determinants constitute a regular sequence. Thus, $I(\phi_3)$ contains a regular 
 sequence of length 3.

The other cases (ii)-(1) and (2) can be done similarly and are, hence, omitted.
\end{proof}

\vspace{5mm}\noindent\textbf{Case 3(a) :} Suppose that $I_S$ is minimally generated by the polynomials
$$f_{1}=x_{1}^{\alpha_{1}}-x_{2}^{\alpha_{12}}x_{4}^{\alpha_{14}},
\hspace{7mm}f_{2}=x_{2}^{\alpha_{2}}-x_{1}^{\alpha_{21}}x_{3}^{\alpha_{23}},
\hspace{7mm}f_{3}=x_{3}^{\alpha_{3}}-x_{1}^{\alpha_{31}}x_{4}^{\alpha_{34}}, $$\\[-9mm]
$$f_{4}=x_{4}^{\alpha_{4}}-x_{2}^{\alpha_{42}}x_{3}^{\alpha_{43}},
\hspace{7mm}f_{5}=x_{1}^{\alpha_{31}}x_{2}^{\alpha_{42}}-x_{3}^{\alpha_{23}}x_{4}^{\alpha_{14}}$$

\noindent where $\alpha_{1}=\alpha_{21}+\alpha_{31}$, \; $\alpha_{2}=\alpha_{12}+\alpha_{42}$, \;
 $\alpha_{3}=\alpha_{23}+\alpha_{43}$,  \; $\alpha_{4}=\alpha_{14}+\alpha_{34}.$
The condition $n_1<n_2<n_3<n_4$ implies $\alpha_1>\alpha_{12}+\alpha_{14}$ and 
$ \alpha_4<\alpha_{42}+\alpha_{43}$.

$C_S$ has Cohen-Macaulay tangent cone only under the conditions 
$\alpha_2 \leq \alpha_{21}+\alpha_{23}$  and $\alpha_3 \leq \alpha_{31}+\alpha_{34}$ 
by \cite{arslan-katsabekis-nalbandiyan}. Mete and Zengin   \cite{mete-zengin} gave the 
explicit minimal free resolution for the tangent cone of $C_S$ and showed that the Betti 
sequence of its tangent cone is $(1,5,6,2)$.

\vspace{5mm}\noindent\textbf{Case 3(b) :} Suppose that $I_S$ is minimally generated by the polynomials
$$f_{1}=x_{1}^{\alpha_{1}}-x_{2}^{\alpha_{12}}x_{4}^{\alpha_{14}},
\hspace{7mm}f_{2}=x_{2}^{\alpha_{2}}-x_{3}^{\alpha_{23}}x_{4}^{\alpha_{24}},
\hspace{7mm}f_{3}=x_{3}^{\alpha_{3}}-x_{1}^{\alpha_{31}}x_{2}^{\alpha_{32}},$$ \\[-9mm]
$$f_{4}=x_{4}^{\alpha_{4}}-x_{1}^{\alpha_{41}}x_{3}^{\alpha_{43}}, 
\hspace{7mm}f_{5}=x_{2}^{\alpha_{12}}x_{3}^{\alpha_{43}}-x_{1}^{\alpha_{31}}x_{4}^{\alpha_{24}}.$$

\noindent Here, $\alpha_{1}=\alpha_{31}+\alpha_{41}$, \; $\alpha_{2}=\alpha_{12}+\alpha_{32}$, \;
 $\alpha_{3}=\alpha_{23}+\alpha_{43}$,  \; $\alpha_{4}=\alpha_{14}+\alpha_{24}.$
The condition $n_1<n_2<n_3<n_4$ implies
$\alpha_1>\alpha_{12}+\alpha_{14}$, \; $\alpha_2>\alpha_{23}+\alpha_{24}$, \; $\alpha_3 < \alpha_{31}+\alpha_{32}$ 
and $ \alpha_4<\alpha_{41}+\alpha_{43}$.

\begin{theorem}\label{thm3.5} The Betti sequence of the tangent cone $R/I_{S_*}$ of $C_S$ is $(1,6,8,3)$, if $I_S$ is 
given as in the Case 3(b).
\end{theorem}
 
\begin{proof} Suppose that $I_S$ is given as in the Case 3(b). 

\noindent (i) By Proposition 2.21 in \cite{katsabekis}, if $\alpha_{23} < \alpha_{43}$ and $\alpha_{14} \leq \alpha_{24}$, then
$$\{f_1,f_2,f_3,f_4,f_5,f_{6}=x_{2}^{\alpha_{2}+\alpha_{12}}-x_1^{\alpha_{1}}x_{3}^{\alpha_{23}}x_{4}^{\alpha_{24}-\alpha_{14}}\},$$

\noindent (ii) By Proposition 2.23 in \cite{katsabekis},

\indent (1) if $\alpha_{43} \leq \alpha_{23}$ and $\alpha_{24} < \alpha_{14}$ then 
$$\{f_1,f_2,f_3,f_4,f_5,f_{6}=x_{2}^{\alpha_{2}+\alpha_{12}}-x_1^{\alpha_{31}}x_{3}^{\alpha_{23}-\alpha_{43}}x_{4}^{2\alpha_{24}}\},$$

\indent (2) if $\alpha_{43} \leq \alpha_{23}$ and $\alpha_{14} \leq \alpha_{24}$ then
$$\{f_1,f_2,f_3,f_4,f_5,f_{6}=x_{2}^{\alpha_{2}+\alpha_{12}}-x_1^{\alpha_{1}}x_{3}^{\alpha_{23}}x_{4}^{\alpha_{24}-\alpha_{14}}\}$$
\noindent are  standard bases for $I_S$. Since $\overline{I}=\pi_{i}(I_{S_*})$ which sends $x_1$ to 0, then the generators of the 
defining ideal of $\overline{I}$ is generated by 
$$G_{*}=(x_2^{\alpha_{12}}x_4^{\alpha_{14}},   x_3^{\alpha_{23}}x_4^{\alpha_{24}}, x_3^{\alpha_{3}}, x_4^{\alpha_{4}}, x_2^{\alpha_{12}}x_3^{\alpha_{43}},x_2^{\alpha_{2}+\alpha_{12}}).$$

\noindent  Now, consider the case (i). Since the Betti sequences of   $R/I_{S_*}$ and $\overline{R}/\overline{I}$ are the same which 
follows from Lemma \ref{lemma3.1}, we will show that the sequence,
\begin{center}
$ 0 \rightarrow R^3 \xrightarrow{\phi_3} R^8 \xrightarrow{\phi_2} R^6\xrightarrow{\phi_1} R^1  \rightarrow 0 $
\end{center}
\noindent is a minimal free resolution of $\overline{R}/\overline{I}$, where

{\footnotesize
$$
\phi_1=
\begin{pmatrix}
x_2^{\alpha_{12}}x_4^{\alpha_{14}} &
x_3^{\alpha_{23}}x_4^{\alpha_{24}}&
x_3^{\alpha_{3}}&
x_{4}^{\alpha_{4}}&
x_{2}^{\alpha_{12}}x_{3}^{\alpha_{43}}&
x_{2}^{\alpha_{2}+\alpha_{12}} 
\end{pmatrix},
$$\\[-7mm]

$$
\phi_2=
\begin{pmatrix}
x_3^{\alpha_{43}}  & x_3^{\alpha_{23}} x_4^{\alpha_{24}-\alpha_{14}} &   x_4^{\alpha_{24}} & x_2^{\alpha_{2}}  &  0  & 0  & 0 &  0 \\[0.5mm]
0 & -x_2^{\alpha_{12}} & 0  & 0 &  0  &  0  & -x_3^{\alpha_{43}}  & x_4^{\alpha_{14}}  \\[0.5mm]
0  &  0  & 0 & 0  & -x_2^{\alpha_{12}} &  0 & x_4^{\alpha_{24}}   & 0 \\[0.5mm]
0 & 0 & -x_2^{\alpha_{12}} & 0 & 0  & 0 & 0 &  -x_3^{\alpha_{23}}  \\[0.5mm]
-x_4^{\alpha_{14}} & 0 & 0  & 0 & x_3^{\alpha_{23}}  &  x_2^{\alpha_{2}} & 0 & 0  \\[0.5mm]
0 & 0 &  0  & -x_4^{\alpha_{14}}  & 0 & -x_3^{\alpha_{43}} & 0 & 0  \\
\end{pmatrix},
$$\\[-7mm]

$$
\phi_3=
\begin{pmatrix}
x_{2}^{\alpha_{2}} & 0 & x_{3}^{\alpha_{23}}x_4^{\alpha_{24}-\alpha_{14}}\\[0.5mm]
0  &   x_{4}^{\alpha_{14}}  & -x_3^{\alpha_{43}} \\[0.5mm]
0 &  -x_{3}^{\alpha_{23}}  & 0 \\[0.5mm]
-x_{3}^{\alpha_{43}} & 0  & 0 \\[0.5mm]
0  &  0 & x_{4}^{\alpha_{24}} \\[0.5mm]
x_{4}^{\alpha_{14}} & 0 &  0 \\[0.5mm]
0 & 0 & x_{2}^{\alpha_{12}}   \\[0.5mm]
0 & x_{2}^{\alpha_{12}}  & 0
\end{pmatrix}.
$$
}

\noindent Since $\phi_1\phi_2=\phi_2\phi_3=0$, the sequence above is a complex. 
It is easy to show that $rank(\phi_1)=1$, $rank(\phi_2)=5$ and $rank(\phi_3)=3$. 
As the 5-minors of $\phi_2$, we have $-x_3^{2\alpha_{3}+\alpha_{43}}$ by removing
the row 3 and the columns 2, 3, 4, and $x_4^{2{\alpha_{4}+\alpha_{14}}}$ by removing
the row 4 and  the columns 2, 5, 6. These two determinants are relatively prime, so 
$I(\phi_2)$ contains a regular sequence of length 2. As the 3-minors of $\phi_3$, we get 
$-x_2^{\alpha_{2}+2\alpha_{12}}$ by removing the rows 2, 3, 4, 5, 6, and 
$x_3^{\alpha_{3}+\alpha_{43}}$ by removing the rows 1, 5, 6, 7, 8,
and finally, $x_4^{\alpha_{4}+\alpha_{14}}$ by removing the rows 1, 3, 4, 7, 8. These three 
determinants constitute a regular sequence. Thus, $I(\phi_3)$ contains a regular sequence 
of length 3.

The other cases (ii)-(1) and (2) can be done similarly and are, hence, omitted.
\end{proof}

\bibliographystyle{amsplain}

\end{document}